\documentclass[12pt]{article}
\usepackage[all]{xy}
\usepackage{pict2e, graphicx}
\usepackage{amsmath}
\usepackage{amssymb}
\usepackage{latexsym}

\usepackage{setspace}

\newcommand{\cG}{{\cal G}}

\newcommand{\Z}{\mathbb{Z}}

\newcommand{\floor}[1]{\left \lfloor #1 \right \rfloor}
\newcommand{\ceil}[1]{\left \lceil #1 \right \rceil}

\newcommand{\qed}[0]{\begin{flushright} \rule{2mm}{3mm} \end{flushright}}
\newenvironment{proof}{\noindent{\bf Proof: }}{{\qed}}

\newcommand{\beq}[1]{\begin{equation}\label{#1}}
\newcommand{\enq}[0]{\end{equation}}

\newtheorem{theorem}{Theorem}[section]
\newtheorem{defn}[theorem]{Definition}

\newtheorem{lemma}[theorem]{Lemma}

\newtheorem{corollary}[theorem]{Corollary}

\newcommand{\figref}[1]{Figure~\ref{#1}}

\title{Revolutionaries and Spies}

\if false

\author{David Howard\thanks{\noindent The research of the first author was
supported by BSF grant no. $2006099$, and by ISF grants Nos.
$779/08$, $859/08$ and $938/06$.
}
 \\Department of Mathematics  \\ Technion-Israel Institute
of Technology \\ Technion City, Haifa 32000\ \ Israel \\ \tt{howard@techunix.technion.ac.il}\\
\and Clifford Smyth \\ Department of Mathematics \\ UNC-Greensboro \\
Greensboro, NC  \ \ USA \\ \tt{cdsmyth@uncg.edu}
}

\fi

\author{David Howard\thanks{\noindent The research of the first author was
supported by BSF grant no. $2006099$, and by ISF grants Nos.
$779/08$, $859/08$ and $938/06$.
}
\\ Department of Mathematics \\ Colgate University\\ Hamilton, NY 13346, U.S.A.
\and Clifford Smyth \\ Department of Mathematics and Statistics\\ UNC-Greensboro \\
Greensboro, NC  27402  USA \\ \tt{cdsmyth@uncg.edu}
}
\date{\today}

\begin{document}

\maketitle

\begin{abstract} \emph{Revolutionaries and Spies} is a game, $\cG(G,r,s,k)$, played on a graph $G$ between two teams: one team consists of $r$ revolutionaries, the other consists of $s$ spies.  To start, each revolutionary chooses a vertex as its position.  The spies then do the same.   (Throughout the game, there is no restriction on the number of revolutionaries and spies that may be positioned on any given vertex.) The revolutionaries and spies then alternate moves with the revolutionaries going first.  To move, each revolutionary simultaneously chooses to stay put on its vertex or to move to an adjacent vertex.  The spies move in the same way.    The goal of the revolutionaries is to place $k$ of their team on some vertex $v$ in such a way that the spies cannot place one of their spies at $v$ in their next move; this is a win for the revolutionaries.  If the spies can prevent this forever, they win.  There is no hidden information; the positions of all revolutionaries and spies is known to both sides at all times.

We will present a number of basic results as well as the result that  if $\cG(\Z^2,r,s,2)$ is a win for the spies, then $s \geq 6 \lfloor \frac{r}{8} \rfloor$. (Here allowable moves in $\Z^2$ consist of one-step horizontal, vertical or diagonal moves.)
\end{abstract}

\section{Introduction}

Let $G$ be a graph, possibly infinite, and let $r,s,k$ be positive integers.  \emph{Revolutionaries and Spies}, denoted $\cG(G,r,s,k)$, is the following game invented by Beck \cite{Bec}. $\cG(G,r,s,k)$ is played between two teams of agents.   The first team is a team of $r$ ``revolutionaries'' and the second a team of $s$ ``spies''.  In round $0$, each revolutionary takes a position on some vertex of $G$ and, afterwards, each spy does the same.  There is no restriction on the number of spies or revolutionaries that may be placed on a vertex at any point in the game.  For $i \geq 1$, round $i$ begins with each revolutionary moving either to a vertex adjacent to its current vertex or staying at its current vertex.  Round $i$ ends with the spies moving in the same fashion.  The revolutionaries have a \emph{meeting} of size $k$ at a vertex $v$ if $k$ or more revolutionaries are present at that vertex. A meeting at vertex $v$ is \emph{guarded} if a spy is present at $v$.  The revolutionaries win $\cG(G,r,s,k)$ if they can guarantee an unguarded meeting of size $k$ at the \emph{end} of some round after the spies have had one move to respond.  Otherwise the spies have a strategy to prevent this for all time and we say the spies win $\cG(G,r,s,k)$.  Note that when the revolutionaries first form a number of meetings of size $k$, the spies have one move to guard them all.  If the spies succeed, the game continues; otherwise, the revolutionaries win.

Continuous versions of this problem can also be considered.  For example, one could play the game in the plane, where each agent has the power to move to points within Euclidean distance $1$ from their current position.  Or the agents may move continuously at rates less than or equal to $1$.  These variants were suggested by Beck \cite{Bec}.

There is a similarity between the moves in \emph{Spies and Revolutionaries} and \emph{Cops and Robbers}  \cite{NowWin, Qui}.  The latter game is a pursuit game played by a cop and a robber on a graph $G$: the cop chooses a vertex, then the robber chooses a vertex, and players move alternately starting with the cop. A move consists of either staying at one's present vertex or  moving to an adjacent vertex; each
move is seen by both players. The cop wins if he manages to occupy the same vertex as the robber and the robber wins if he avoids this forever.  The graphs on which the cop has a winning strategy have been characterized \cite{NowWin, Qui}.  

\begin{defn}
Let $G$ be a graph and let $r $ and $k$ be positive integers.  We define $\sigma(G,r,k)$ to be the minimum value of $s$ such that the spies win $\cG(G,r,s,k)$.
\end{defn}

We record the following trivial observation.
\begin{lemma}  \label{splittinglemma} If $G$ is a graph,  then
\[ \min\left\{|V(G)|, \floor{\frac{r}{k}}\right\} \leq \sigma(G,r,k) \leq \min\{|V(G)|, r - k + 1\}\]
\end{lemma}
\begin{proof}
For the lower bound, observe that the revolutionaries can construct $\min\left\{|V(G)|, \floor{\frac{r}{k}}\right\}$ meetings of size $k$ on distinct vertices in their first move. The spies must be able to guard all of these meetings.  The upper bound is obtained by considering two separate strategies for the spies. The spies could win by using $|V(G)|$ spies to permanently guard each vertex. The spies could also win by picking $r-k+1$ revolutionaries to follow, assigning them distinct spies to stay on their respective positions.  The spies win with this strategy as well, as every meeting of size $k$ will involve a revolutionary that is being followed.
\end{proof}
This trivial lower bound from Lemma \ref{splittinglemma} is attained on acyclic  graphs.
\begin{theorem}  \label{acyclicthm}  If $G$ is an acyclic graph, then
\[ \sigma(G,r,k) = \min\left\{|V(G)|, \floor{\frac{r}{k}}\right\}.\]
\end{theorem}

This theorem was initially proved by one of the authors in an equivalent form in \cite{Smy}.  The proof of Theorem \ref{acyclicthm} will appear in  \cite{CraSmyWes} which covers \emph{Revolutionaries and Spies} on trees and unicyclic graphs.  It would be interesting to characterize those graphs $G$ for which $\sigma(G,r,k) = \min\left\{|V(G)|, \floor{\frac{r}{k}}\right\}$.  By Theorem \ref{powerlemma} all powers of acyclic graphs are included.

For $v,w \in V(G)$, let $d_G(v,w)$ be the \emph{distance} between $v$ and $w$ in $G$, i.e. the minimum length of an $v$,$w$-path in $G$.  If no such path exists, then let $d_G(v,w) = +\infty$.  Note $d_G(v,v) = 0$.

Let $G$ and $H$ be graphs.  The \emph{strong product} of $G$ and $H$, denoted $G \boxtimes H$, is the graph with vertex set $V(G) \times V(H)$ such that vertices $(g,h)$ and $(g',h')$ are adjacent  in $G \boxtimes H$ if and only if  $(g,h) \neq (g',h')$,  $d_G(g,g') \leq 1$, and $d_H(h,h') \leq 1$.  We denote by $\Z$ the graph $G$ with $V(G) = \Z$ and $E(G) = \{\{i,i+1\}: i \in \Z\}$.  For $d \geq 1$, let $\Z^{\boxtimes d}$ be the $d$-fold strong product of $\Z$ with itself.  

We study \emph{Revolutionaries and Spies} on $\Z^{\boxtimes d}$.  Perhaps one of the most basic (yet nontrivial) quantities to study is the threshold $\sigma(\Z^{\boxtimes d}, r, 2)$.
\begin{theorem} \label{Z2thm} If $d$ is an integer with $d \geq 2$, then $\sigma(\Z^{\boxtimes d}, r, 2) \geq 6 \floor{\frac{r}{8}}$.
\end{theorem}
Note that Lemma \ref{splittinglemma} only implies $\sigma(\Z^{\boxtimes d}, r, 2) \geq \floor{\frac{r}{2}}$.   We conjecture that $\sigma(\Z^{\boxtimes d}, r,2) = r-2$ for $d \geq 2$.


While this manuscript was under review the authors learned of the work in progress of several other authors.

Theorem \ref{acyclicthm} has been generalized in  \cite{CraSmyWes} to the case of unicyclic graphs $G$: $\sigma(G,r,k) \in \{\floor{\frac{r}{k}},\ceil{\frac{r}{k}}\}$ and furthermore if $k \nmid r$, then $\sigma(G,r,k) = \floor{\frac{r}{k}}$ if and only if $\ell \leq \max \{ \floor{\frac{r}{k}} - t + 2,3\}$ where $\ell$ is the length of the cycle and $t = |V(G)| - \ell$.

\emph{Revolutionaries and Spies} has now been studied on  interval graphs, hypercubes, random graphs, complete multipartite graphs, and also on a common generalization of trees and graphs with a dominating vertex \cite{ButCraPulWesZam}. 

The organization of the paper is as follows:  we present a number of basic definitions and results in Section 2 and then prove Theorem \ref{Z2thm} in Section 3.


\section{Basic Results}

Note that if $r',r,k'$, and $k$ are integers with $r' \geq r \geq 1$ and $k' \geq k \geq 1$, then $\sigma(G,r',k) \geq \sigma(G,r,k)$ and $\sigma(G,r,k') \leq \sigma(G,r,k)$.  For example, if $\cG(G,r',s,k)$ is a win for the spies, then so is $\cG(G,r,s,k)$:  the spies use their winning strategy in $\cG(G,r',s,k)$ to play $\cG(G,r,s,k)$ as if there were an additional $r'-r$ revolutionaries fixed on some vertex.

\begin{lemma} \label{sumlemma}

Let $k = \sum_{i=1}^c k_i$ and $r = \sum_{i=1}^c r_i$, where all quantities are positive integers.  For any graph $G$,

\[\sigma(G,r,k-c+1) \leq \sum_{i=1}^c \sigma(G,r_i,k_i)\]

\end{lemma}

\begin{proof}

Let $s_i = \sigma(G,k_i,r_i)$ for $1 \leq i \leq c$, and set $s = \sum_{i=1}^c s_i$.  We claim $\cG(G,r,s,k-c+1)$ is a win for the spies.  The spies divide the revolutionaries into $c$ disjoint groups $R_i$, $1 \leq i \leq c$, where $R_i$ consists of $r_i$ revolutionaries. The spies are also divided into disjoint groups $S_i$, $1 \leq i \leq c$, where $S_i$ consists of $s_i$ spies. For $1 \leq i \leq c$ the spies simultaneously use group $S_i$ to prevent a meeting of size $k_i$ amongst the revolutionaries in group $R_i$.  The largest possible unguarded meeting under this strategy is $\sum_{i=1}^c (k_i - 1) = k-c$.
\end{proof}
Note that if $r, a,k$ are positive integers, then Lemma \ref{sumlemma} implies the three inequalities $\sigma(G,r+a,k+a) \leq \sigma(G,r,k)$, $\sigma(G,r+a,k) \leq \sigma(G,r,k) + a$, and $\sigma(G,a r,a k) \leq a \sigma(G,r,k)$.

\begin{lemma}
If $G$ is a graph and $a$ is a positive integer, then $\sigma(G,a r, a k) \geq \sigma(G,r,k)$.
\end{lemma}
\begin{proof}
Let $s = \sigma(G, a r, a k)$.  We claim that $\cG(G,r,s,k)$ is a win for the spies.  The spies identify each revolutionary with a group of $a$ revolutionaries sharing the same position and then follow their winning strategy in $\cG(G,a r,s, a k)$.  The spies prevent a meeting of size $k$ in $\cG(G,r,s,k)$ in this way because the corresponding groups will form a meeting of size $a k$ in $\cG(G,a r, s, ak)$. 
\end{proof}

If $G$ is a graph, the $n$th \emph{power} of $G$ is the graph $G^n$ with $V(G^n) = V(G)$ and $E(G^n) = \{\{v,w\}: 0 < d_G(v,w) \leq n\}$.

\begin{theorem} \label{powerlemma}
If $G$ is a graph, then $\sigma(G^n,r,k) \leq \sigma(G,r,k)$.
\end{theorem}
\begin{proof}
Let $s = \sigma(G,r,k)$.  We show that $\cG(G^n,r,s,k)$ is a win for the spies.  Let $R_0$ be the initial position of the revoluationaries.  The spies play the position $S_0$ from their winning strategy in $\cG(G,r,s,k)$.  If the revolutionaries move to a new position $R_n$ in $G^n$, then there are intermediate positions $R_1, \ldots, R_{n-1}$ such that $R_i$ is one move from $R_{i-1}$ in $\cG(G,r,s,k)$.  Let $S_1, \ldots, S_{n}$ be the corresponding countermoves for the spies in their winning strategy in $\cG(G,r,s,k)$.  Since $S_n$ is one move from $S_0$ in $G^n$, the spies may play $S_n$ as their response to $R_n$ in $\cG(G^n,r,s,k)$ and the process repeats. This strategy indefinitely prevents a meeting of size $k$.
\end{proof}
\begin{theorem} \label{productlemma} If $G$ and $H$ are graphs, then \[\sigma(G \boxtimes H,r,k) \geq \max(\sigma(G,r,k),\sigma(H,r,k)).\]
\end{theorem}

\begin{proof} When playing the game ${\cal G}(G \boxtimes H,r,s,k)$, the revolutionaries can choose to play only in a single copy of $G$ (or $H$).  Regardless of where the spies play, the revolutionaries can project the spies' positions into the chosen copy of $G$.  If the revolutionaries can win ${\cal G}(G,r,s,k)$ against the projected spies, then they can win in $G \boxtimes H$.  Thus $\sigma(G \boxtimes H,r,s) \geq \sigma(G,r,s)$.  The same argument works for $H$.  Thus, the theorem follows.

\end{proof}

\section{Proof of Theorem \ref{Z2thm}.}

As a warmup, we prove
\begin{theorem} \label{Zthm}
If $r$ and $k$ are positive integers, then $\sigma(\Z,r,k)  = \floor{\frac{r}{k}}$.
\end{theorem}
\begin{proof}  For $1 \leq i \leq r$, let $x_i$ be the position of the $i$th revolutionary in $\Z$ and let $x'_i$ be its position one move later.  Let $x_{(1)} \leq \cdots \leq x_{(r)}$ be the order statistics of the sequence $x$, i.e. a rearrangement of the $x_i$ into non-decreasing order.  Let $x'_{(1)} \leq \cdots \leq x'_{(r)}$ be the order statistics of the sequence $x'$.  Let $s = \floor{\frac{r}{k}}$.  The spies' strategy is to place the $i$th spy at position $y_i = x_{(ik)}$ for $1 \leq i \leq s$.  This clearly prevents an unguarded meeting of size $k$.  If the spies are then able to move to the positions $y'_i = x'_{(ik)}$, this strategy will be sustainable. 

It suffices to show that for all $1 \leq i \leq r$, $|x_{(i)} - x'_{(i)}| \leq 1$.  If $x_{(i)} = x_0$,  then there are at least $i$ revolutionaries on integers less than or equal to $x_0$.  After one move these same $i$ revolutionaries are on positions $x_0+1$ or less, meaning $x'_{(i)} \leq x_0 + 1 = x_{(i)}+1$.  Similarly  $x_{(i)} \leq x'_{(i)}+1$.

\end{proof}

Recall that $a, b \in \Z^{\boxtimes d}$ are adjacent if and only $a \neq b$ and $|a_i-b_i| \leq 1$ for all $1 \leq i \leq d$.

\begin{lemma}
\label{Zdlemma}
If $r$ and $k$ are positive integers with $1 \leq k \leq r/2$,  then $\sigma(\Z^{\boxtimes d},r,k) \leq r-2k+2$.
\end{lemma}
\begin{proof}
Clearly $\sigma(\Z^{\boxtimes d},2k-1,k) \geq 1$.  We give a strategy for the spies showing that $\sigma(\Z^{\boxtimes d},2k-1,k) = 1$.  Suppose the revolutionaries are in some position.  Fix $1 \leq i \leq d$.  Let $x_{i,1}, \ldots, x_{i,2k-1}$ be the $i$th coordinates of the revolutionaries.  Let $c_i= x_{i,(k)}$ be their $k$th order statistic. The spies'  response is to move the spy to the vertex $c = (c_1, \ldots, c_d)$.  By the argument in Theorem \ref{Zthm}, this is a playable strategy for the spies.  Furthermore, it guards all meetings of size $k$ or more.   Lemma \ref{sumlemma} implies $\sigma(\Z^{\boxtimes d},r,k) \leq \sigma(\Z^{\boxtimes d},r-2k+1,1)+ \sigma(\Z^{\boxtimes d},2k-1, k) \leq (r-2k+1) +1 = r - 2k +2$. 
\end{proof}

Theorem \ref{Z2thm} will follow from this next theorem (see Corollary \ref{maincor}).
\begin{theorem} \label{85thm}
We have $\sigma(\Z^{\boxtimes 2},8,2)=6$.
\end{theorem}

\begin{proof}
By Lemma \ref{Zdlemma}, $\sigma(\Z^{\boxtimes 2},8,2) \leq 6$. To complete the proof, we give a winning strategy for the revolutionaries in $\cG(\Z^{\boxtimes 2},8,5,2)$.  In the first round, the revolutionaries place their agents on the eight positions $(\pm 1, \pm 1)$ and $(\pm 3, \pm 3)$ (see \figref{Initial}).  In all of our figures the center point is position $(0,0)$, and each $X$ represents a revolutionary while each $O$ represents a spy.  If two revolutionaries may reach a vertex in $n$ rounds, then there must be a spy within distance $n$ of $v$ to guard that potential meeting.  We often describe this by saying that a spy must guard a meeting at $v$ in $n$ rounds.

\noindent \emph{Claim 1 (Box Property): One or more spies must begin in each of the four boxes  $[1,3]\times[1,3]$, $[1,3]\times[-3,-1]$,$[-3,-1]\times[-3,-1]$, $[-3,-1]\times[1,3]$; we call these boxes $B_1$, $B_2$, $B_3$, and $B_4$ respectively (see \figref{Initial}).}

By symmetry, we consider $B_1$.  The revolutionaries at $(3,3)$ and $(1,1)$ may form a meeting at $(2,2)$ in one round; at least one spy must be in $B_1$ to guard this meeting.

Let $W_1$ be the ``wedge" of points $W_1 = \{(x,y): y \geq 1, y \geq |x|\}$.  We also consider the wedges obtained from $W_1$ by reflections in the lines $y=x$ and $y=-x$; in clockwise order from $W_1$, we call these $W_2$, $W_3$, and $W_4$ (see \figref{Initial}).

\noindent \emph{ Claim 2 (Wedge Property): There must be at least two spies present in $W_1$; furthermore one of those spies must be distance $1$ from $(0,2)$ and another distinct spy must be at distance $3$ from $(0,6)$.  We call this the \emph{wedge property} for $W_1$.  By symmetry (reflections through the lines $y=x$ and $y=-x$), analogous wedge properties hold for $W_2$, $W_3$, and $W_4$.}

 By symmetry we consider $W_1$.  The revolutionaries at $(-1,1)$ and $(1,1)$ can form a meeting in one round at $(0,2)$ while, simultaneously, the revolutionaries at $(-3,3)$ and $(3,3)$ can form a meeting at $(0,6)$ in three rounds.  Unless two spies are located as described, one of these meetings will be uncovered.

\begin{figure}
\centering
 \rotatebox{0}{\includegraphics[width=0.5\textwidth]{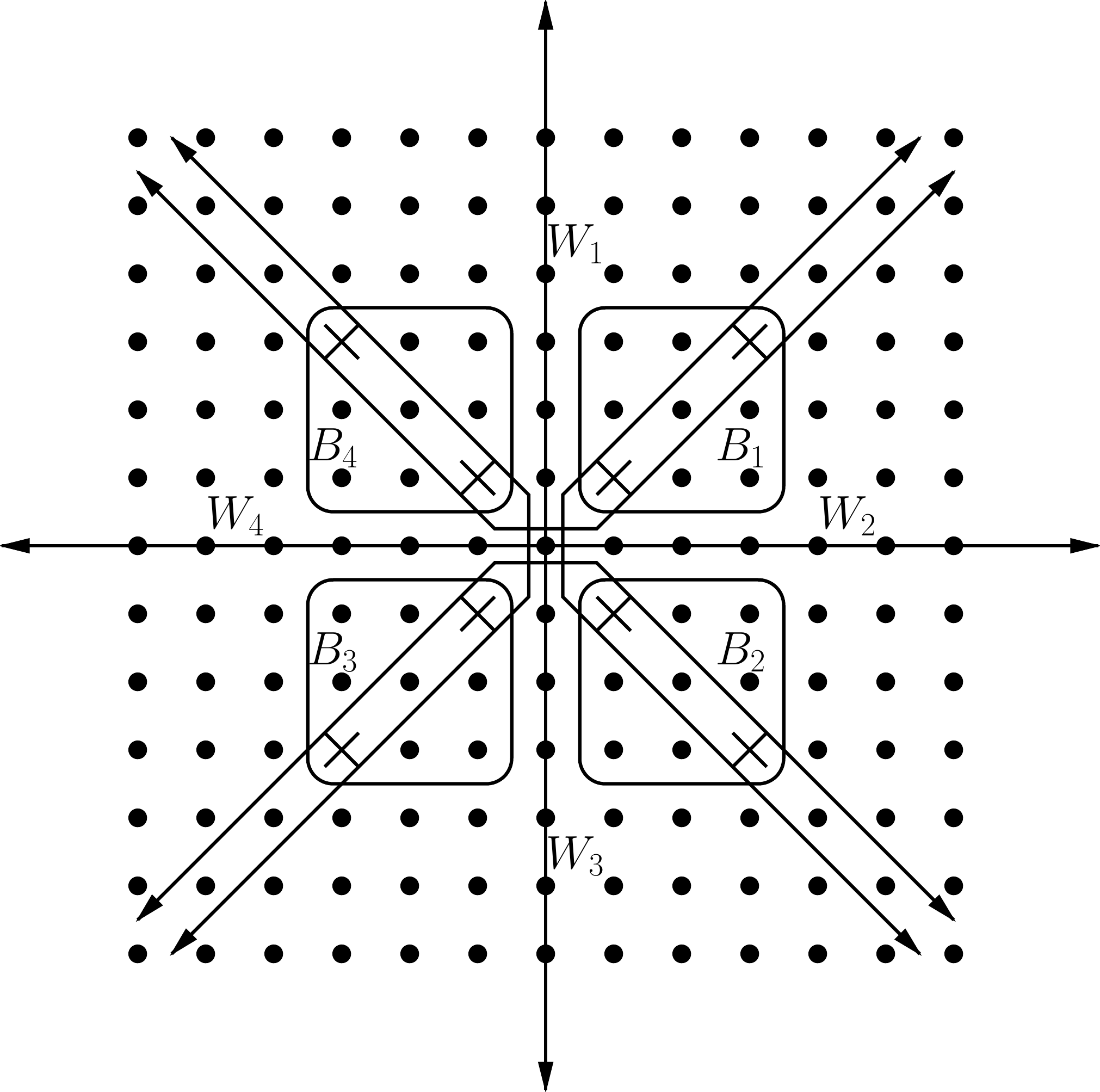}}
\caption{The revolutionaries' initial position, the boxes $B_i$, and wedges $W_i$.}
\label{Initial}
\end{figure}

By symmetry and the pigeonhole principle, we may assume that at least two spies, $s'_1$ and $s''_1$, have positions $(x,y)$ with $x \geq 0$ and $y \geq 0$.  In fact, we may assume that there are exactly two such spies, since each of the boxes $B_2$, $B_3$, and $B_4$ must contain a spy; we call those spies $s_2$, $s_3$, and $s_4$ respectively.  Since the wedge property holds for $W_3$ and $W_4$, spy $s_3$ must lie in both wedges; furthermore its position must be either $(-3,-3)$ or $(-1,-1)$.  This leads to the two cases below.

\noindent {\bf Case 1: $s_3$ is at $(-3,-3)$}

The following discussion is illustrated in \figref{case1}.  Since $s_3$ is at $(-3,-3)$, the wedge property for $W_4$ implies that $s_4$ is in $[-3,-1] \times \{1\}$.  Similarly, $s_2$ is in $\{1\} \times [-3,-1]$.  Furthermore, $s'_1$ must be in $[0,1]\times[0,1]$ to guard against a meeting at $(0,0)$ by the revolutionaries at $(-1,-1)$ and $(1,1)$; neither $s_2$ nor $s_4$ can guard $(0,0)$, since they must guard against the meetings at $(2,-2)$ and $(-2,2)$.   Given these restrictions on $s'_1$, $s_2$, and $s_4$, the wedge properties for $W_1$ and $W_2$ imply that $s''_1$ must be at $(3,3)$.

Suppose that the revolutionaries at $(-1,1)$ and $(-1,-1)$ form a meeting at $(-2,0)$ and those from $(-3,3)$ and $(1,1)$ form a meeting at $(-1,3)$.  Let $(x,y)$ be the position of $s'_1$.  We must have $y \geq 1$ because, since spy $s_4$ must guard $(-2,0)$, $s'_1$ must guard $(-1,3)$.  By symmetry, $x \geq 1$. Thus $s'_1$ is at $(1,1)$.  Now $s_4$ must be in $[-2,-1] \times \{1\}$ to guard against a meeting at $(-1,0)$ of revolutionaries from $(-1,1)$ and $(-1,-1)$.  Similarly $s_2$ must be in $\{1\} \times [-2,-1]$.

We also must have $s_4 = (-1,1)$ or $s_2=(1,-1)$.  If not, then the revolutionaries at $(-1,1)$ and $(-1,-1)$ can meet at $(0,0)$, while the revolutionaries at $(1,1)$ and $(1,-1)$ can meet at $(2,0)$.  It will not be possible for the spies to guard both meetings.  By symmetry we may assume $s_4 = (-1,1)$.

\begin{figure}
\centering
 \rotatebox{0}{\includegraphics[width=0.5\textwidth]{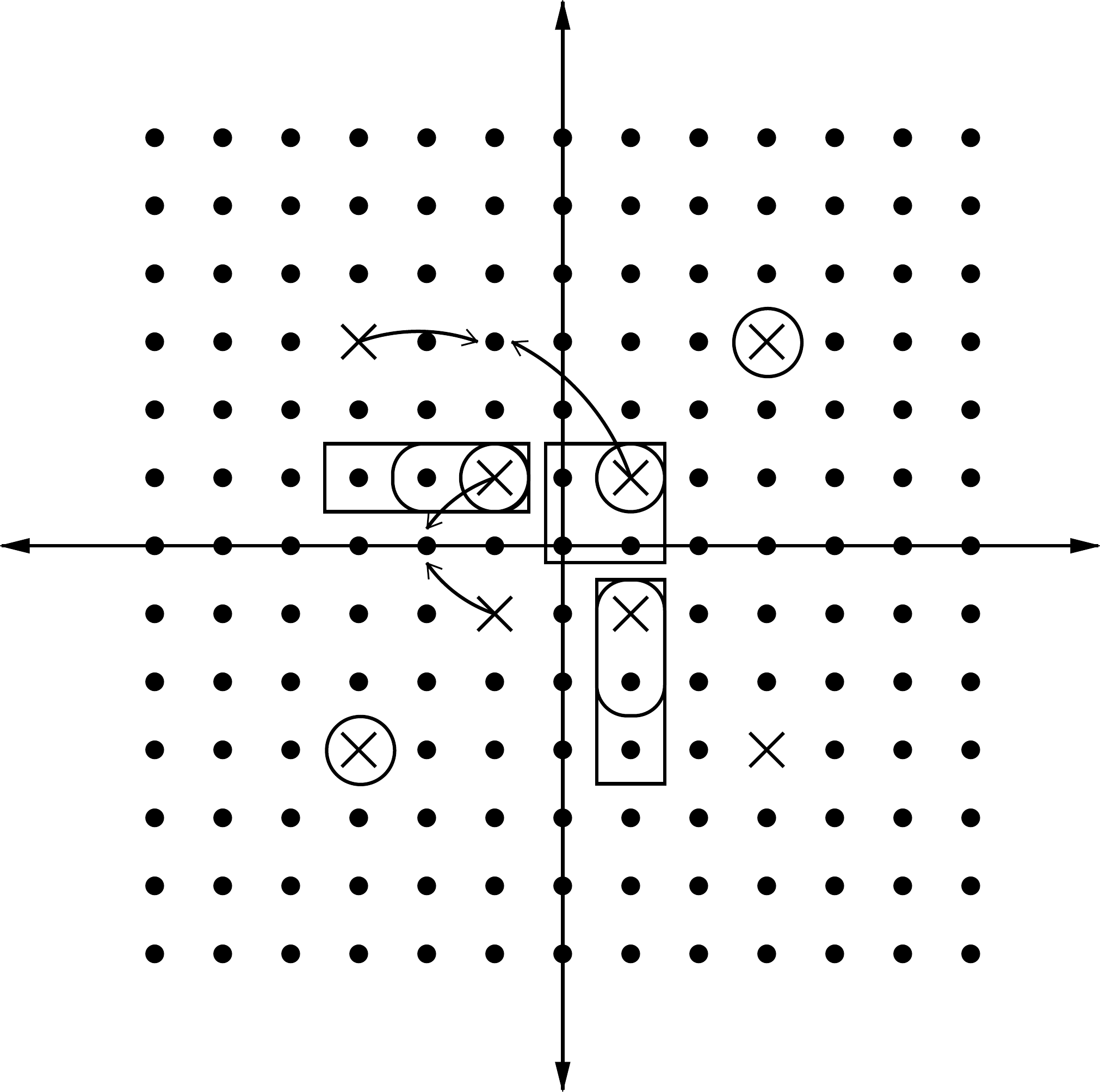}}
\caption{Case $1$.}
\label{case1}
\end{figure}

Thus at the beginning of round 1, we may assume that the spies and revolutionaries are located as in \figref{case1-0}.  This figure indicates that $s_2$ is located somewhere in $R_2 = \{1\} \times [-2,-1]$.  The revolutionaries' strategy is to move the revolutionary at $(-1,-1)$ to $(-2,-2)$ and the one at $(-1,1)$ to $(0,0)$, while keeping the other revolutionaries in place.   

\begin{figure}
\centering
 \rotatebox{0}{\includegraphics[width=0.5\textwidth]{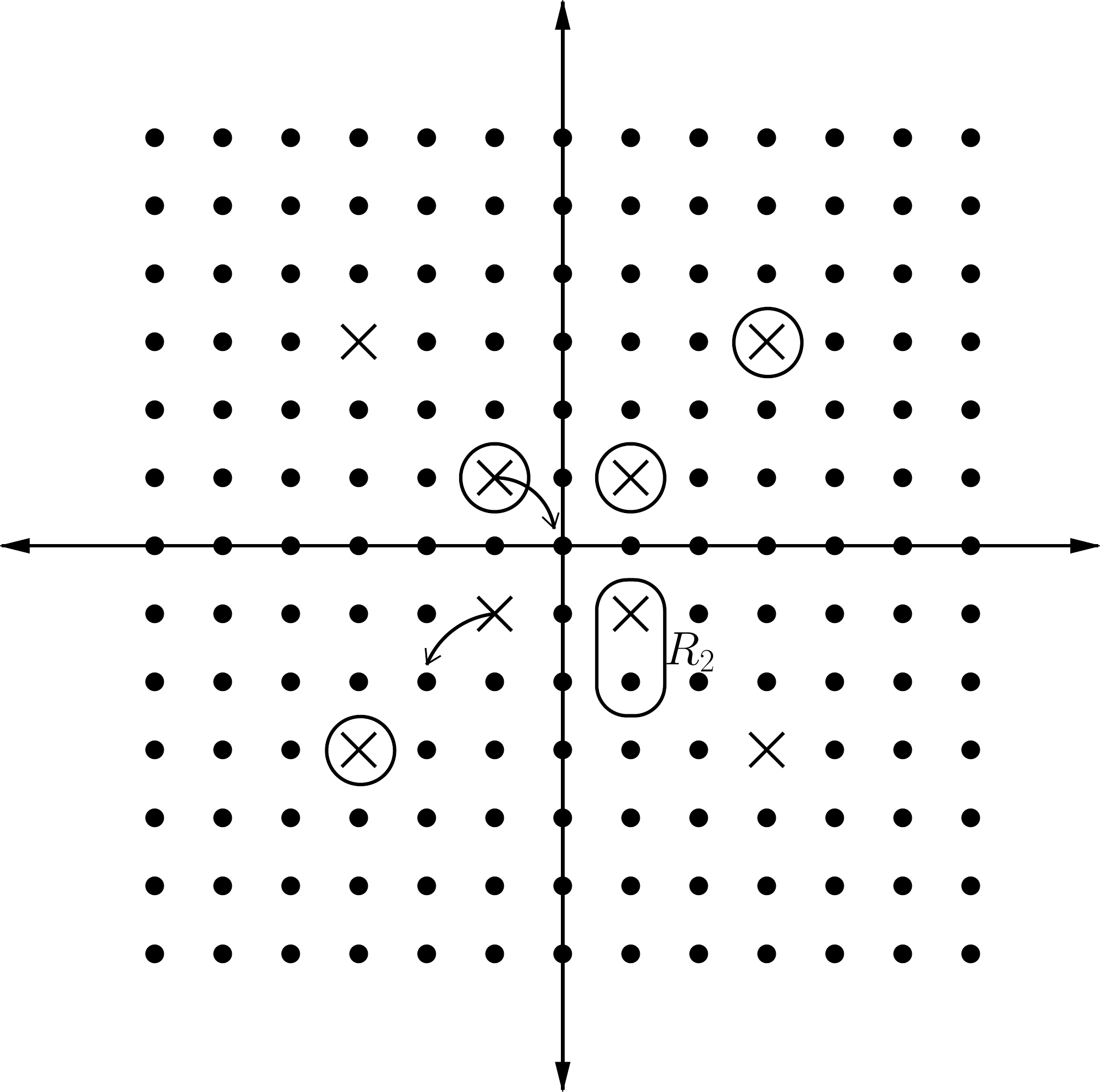}}
\caption{Case $1$: at the end of round $0$.}
\label{case1-0}
\end{figure}

We now analyze the positions that the spies must take at the end of round 1, see \figref{case1-1}. The spies must keep the spies at $(3,3)$ and $(-3,3)$ fixed to continue guarding meetings at $(\pm6, 0)$ and $(0, \pm6)$.  Besides these two spies, only the spy in $R_2$ (and only if it were located at $(1,-2)$) could be moved to help guard $(0,-6)$, but that spy cannot assist, as it must also guard the meeting of the revolutionaries at $(-2,-2)$ and $(1,-1)$ at the point $(0,-3)$.

Let $(a,b)$ be the position of the spy in $R_2$ in round $0$, and let $(a',b')$ be its position in round $1$.  We have $a'=1$, since this spy must guard $(2,-2)$ and also the meeting of $(-2,-2)$ and $(1,-1)$ at $(-1,-3)$.  We must have $b' \leq -2$, since this spy must guard the meeting of the revolutionaries at $(-2,-2)$ and $(3,-3)$ at $(1,-5)$.  Thus, the spies must have a spy in $R'_2 = \{1\} \times [-3,-2]$ at the end of round $1$.  (See Figure \ref{case1-1}.)

The spy located originally at $(1, 1)$ cannot decrease its $x$-coordinate, because it must guard $(2, 0)$. This forces the spy originally located at $(-1, 1)$ to decrease its $y$-coordinate to guard $(-1,-1)$.  This same spy must also decrease its $x$-coordinate to guard against the meeting of the revolutionaries from $(-3,3)$ and $(-2,-2)$ at  $(-5, 1)$. This forces the spy at $(1, 1)$ to move to $(1, 0)$ to guard meetings at $(2, 0)$, $(0, 1)$, and $(0,-1)$. Note that the spy in $R'_2$ cannot help guard these, since it independently must guard the meeting at $(1,-5)$ by revolutionaries from $(-2,-2)$ and $(3,-3)$. 

Now the revolutionaries at $(-3,3)$ and $(1,1)$ can form an unguarded meeting at $(-1, 3)$ in $2$ moves, see \figref{case1-1}.

\begin{figure}
\centering
 \rotatebox{0}{\includegraphics[width=0.5\textwidth]{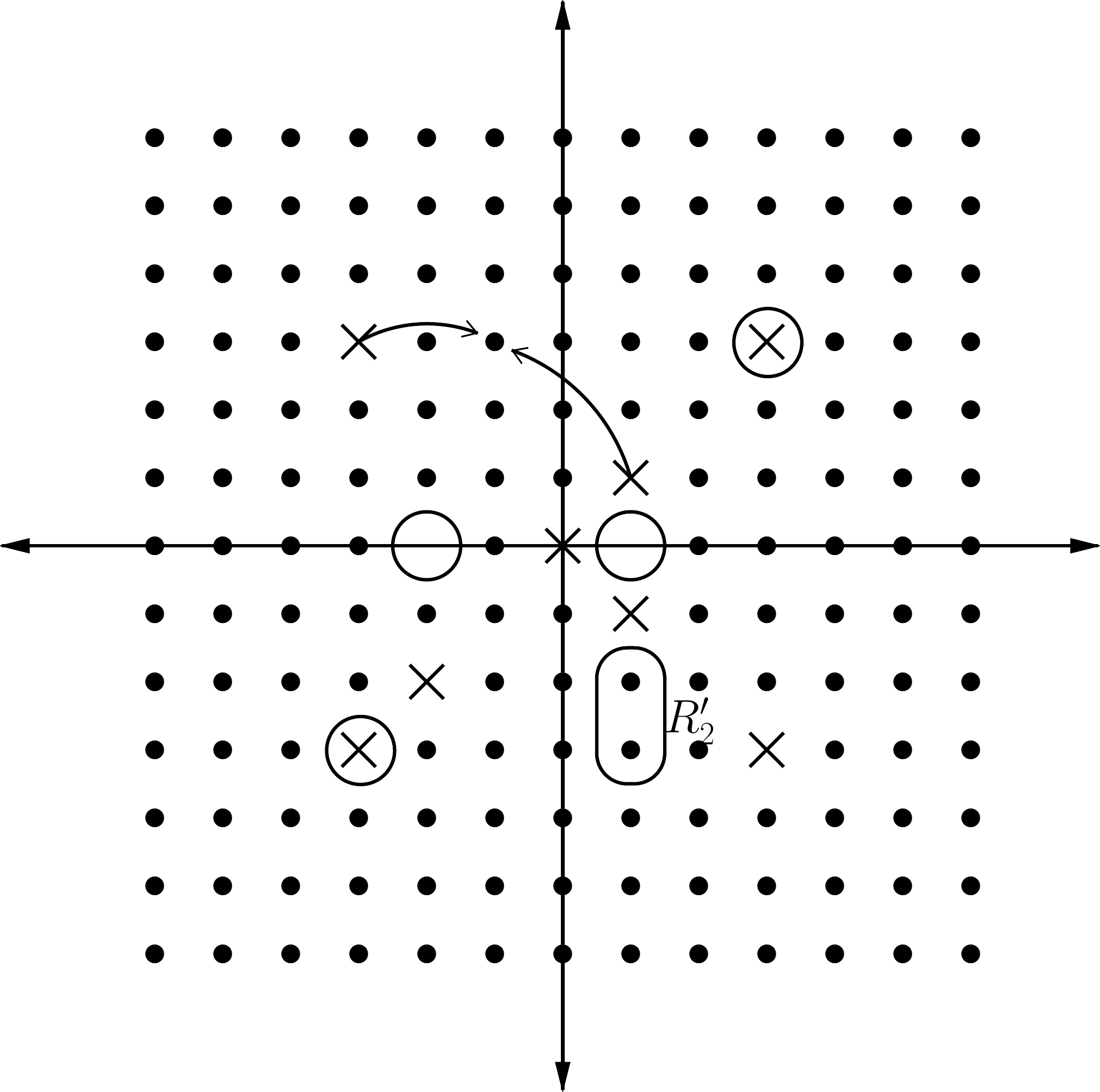}}
\caption{Case $1$: at the end of round $1$.}
\label{case1-1}
\end{figure}

\noindent {\bf Case 2: $s_3$ is at $(-1,-1)$}

Since $s_3$ is at $(-1,-1)$ the wedge property implies that $s_2$ is in $[1,3]\times\{-3\}$ and $s_4$ is in $\{-3\} \times [1,3]$.  If $s_2$ is at $(3,-3)$ and $s_4$ is at $(-3,3)$, then the revolutionaries from $(-1,1)$ and $(-1,-1)$ can form a meeting at $(-2,0)$, which must be guarded by $s_3$.  Simultaneously, the revolutionaries at $(1,-1)$ and $(-3,-3)$ can form a meeting at $(-1,-3)$, which will be unguarded (see \figref{case2a}).

\begin{figure}
\centering
 \rotatebox{0}{\includegraphics[width=0.5\textwidth]{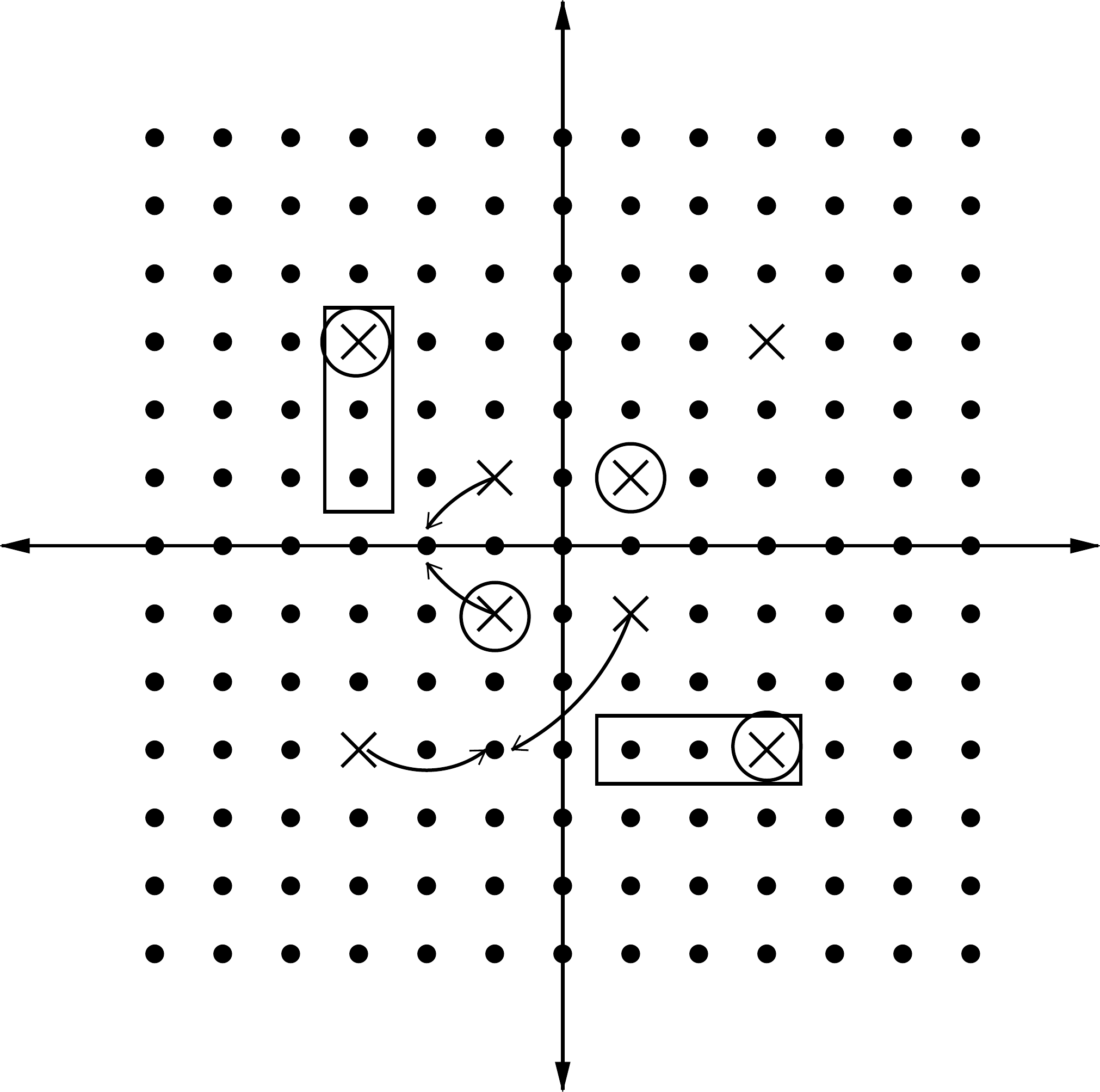}}
\caption{Case 2: at the end of round $0$, Part 1.}
\label{case2a}
\end{figure}

Without loss of generality, suppose instead that $s_2$ is not at $(3,-3)$.  By the wedge property for $W_2$, we must have $s'_1$ is in $[1,3] \times [0,1]$ and $s''_1$ is in $[3,9] \times [0,3]$.  (See \figref{case2b}).  In fact, $s'_1$ must be at $(1,1)$ in order to guard the meeting at $(0,2)$.  Now the revolutionaries at $(-1,-1)$ and $(1,-1)$ can form a meeting at $(0,-1)$.  This must be guarded by $s_3$. Simultaneously, the revolutionaries at $(-1,1)$ and $(-3,-3)$ can form a meeting at $(-3,-1)$ in two rounds; it can only be guarded if $s_4$ began at $(-3,1)$.  By symmetry $s_2$ must begin at $(1,-3)$.  Since now $s''_1$ must guard the meeting at $(0,6)$, it must be located at $(3,3)$, see \figref{case2b}.

\begin{figure}
\centering
 \rotatebox{0}{\includegraphics[width=0.5\textwidth]{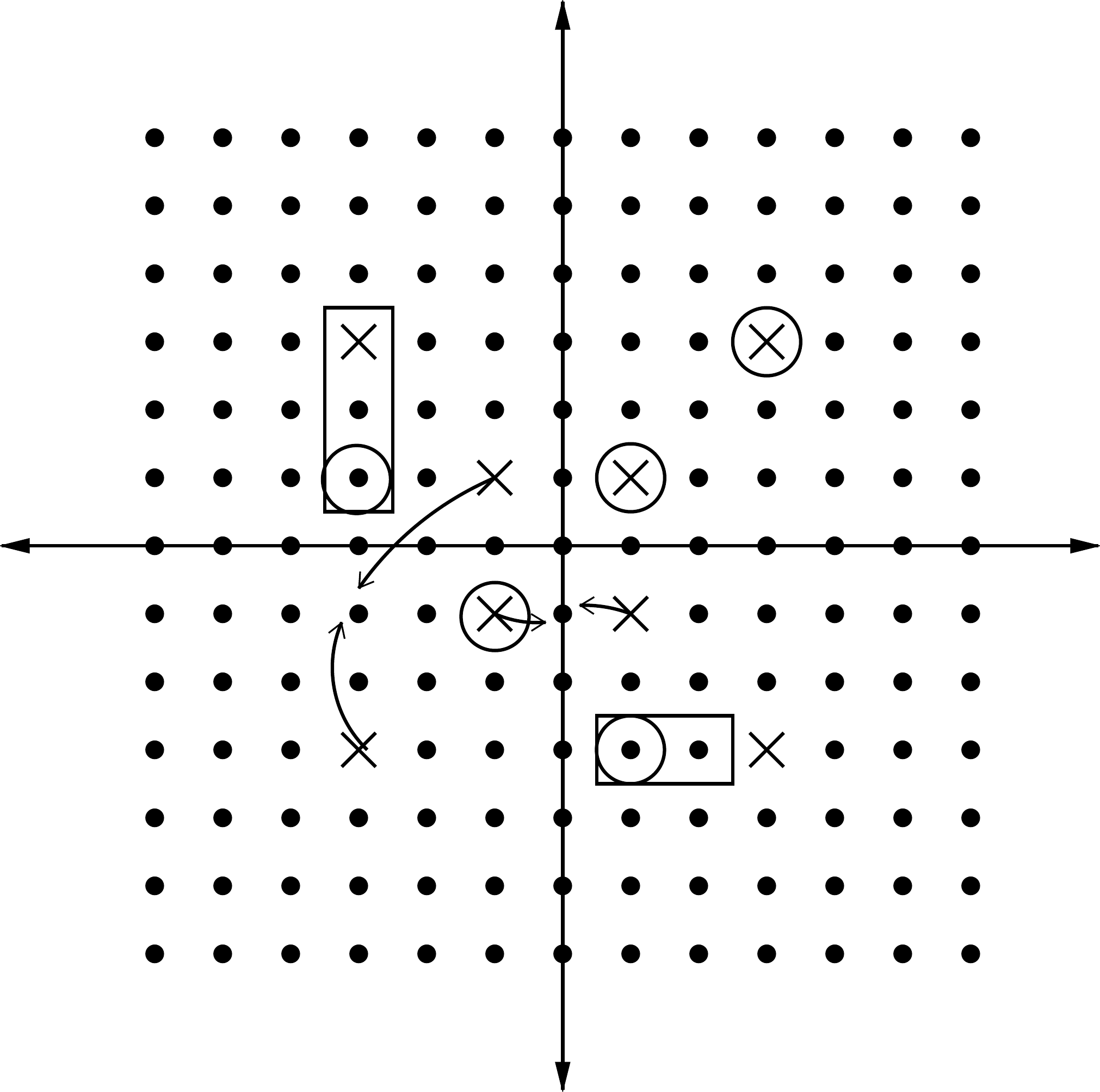}}
\caption{Case 2: at the end of round $0$, Part 2.}
\label{case2b}
\end{figure}

Thus at the end of round 0, the spies and revolutionaries are positioned as in \figref{case2-0}.

\begin{figure}
\centering
 \rotatebox{0}{\includegraphics[width=0.5\textwidth]{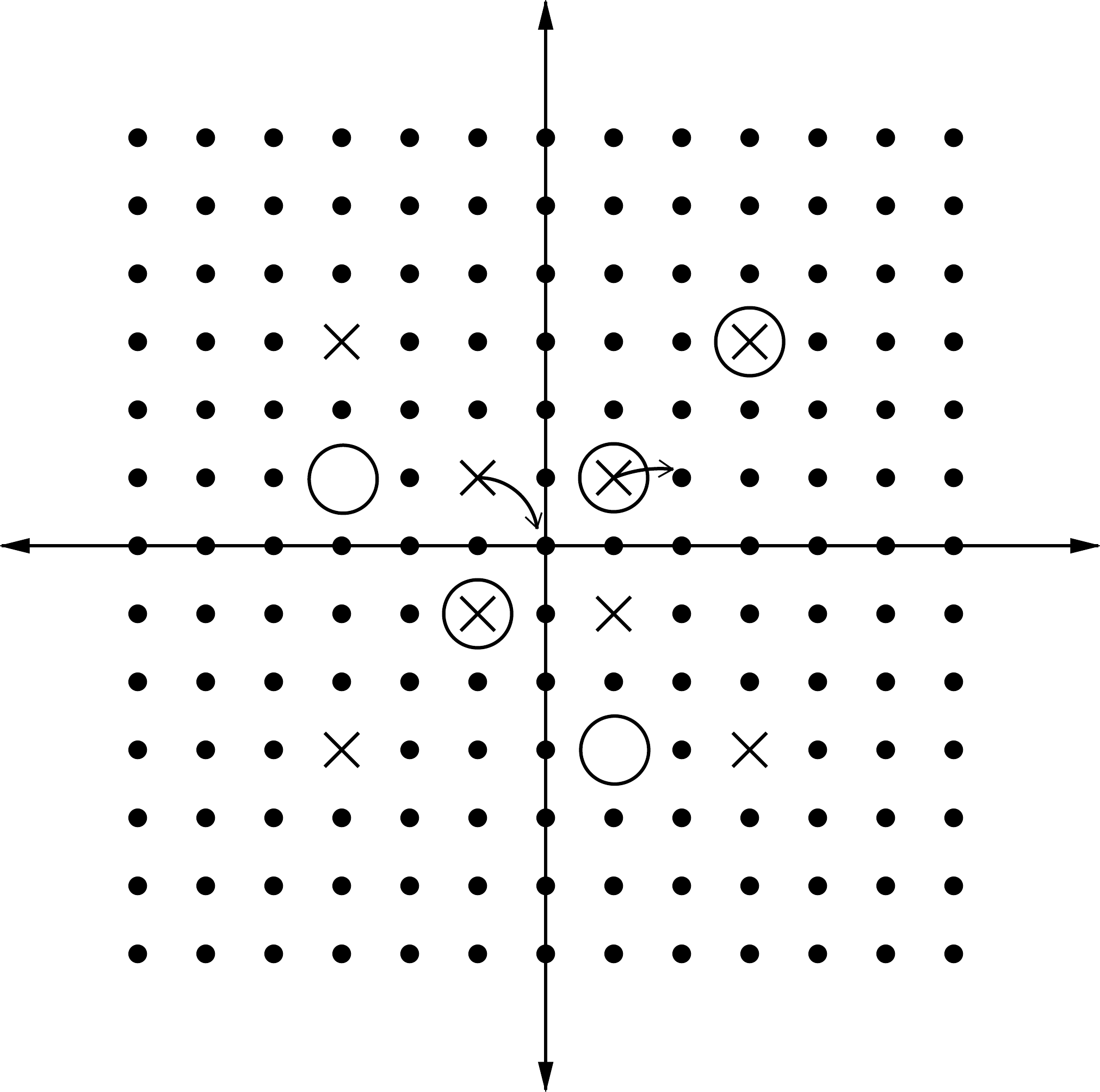}}
\caption{Case 2: at the end of round $0$, Part 3.}
\label{case2-0}
\end{figure}

The revolutionaries' strategy is to move the revolutionary at $(-1,1)$ to $(0,0)$ and the one at $(1,1)$ to $(2,1)$ while leaving all other revolutionaries unchanged. \figref{case2-1} illustrates how the spies must be located at the end of round 1 in order to compensate. The spies must leave the spy at $(3,3)$ in place to guard $(6,0)$ and $(0,6)$.   At the end of its move, the spy at $(-3,1)$ must be somewhere in $L_4 = [-4,-3]\times[0,2]$ as it must guard $(-6,0)$. Similarly, the spy at $(1,-3)$ must remain in $L_2 = [0,2] \times [-4,-3]$.  

The spy at $(-1,-1)$ must stay in place to guard meetings at $(-2,-2)$ and $(0,0)$ in one move. (Note: the spy at $(-1,-1)$ must guard $(0,0)$, since the spy at $(1,1)$ must guard against the potential meeting of the revolutionaries from $(2,1)$ and $(1,-1)$  at the point $(2,0)$.)  Spy $s'_1$ must move to $(1,0)$ to protect against meetings $(0,-1),(1,1),(2,1)$ in the next round. Spy $s_3$ cannot help, since it must protect $(-2,-2)$, see \figref{case2-1}.


\begin{figure}
\centering
 \rotatebox{0}{\includegraphics[width=0.5\textwidth]{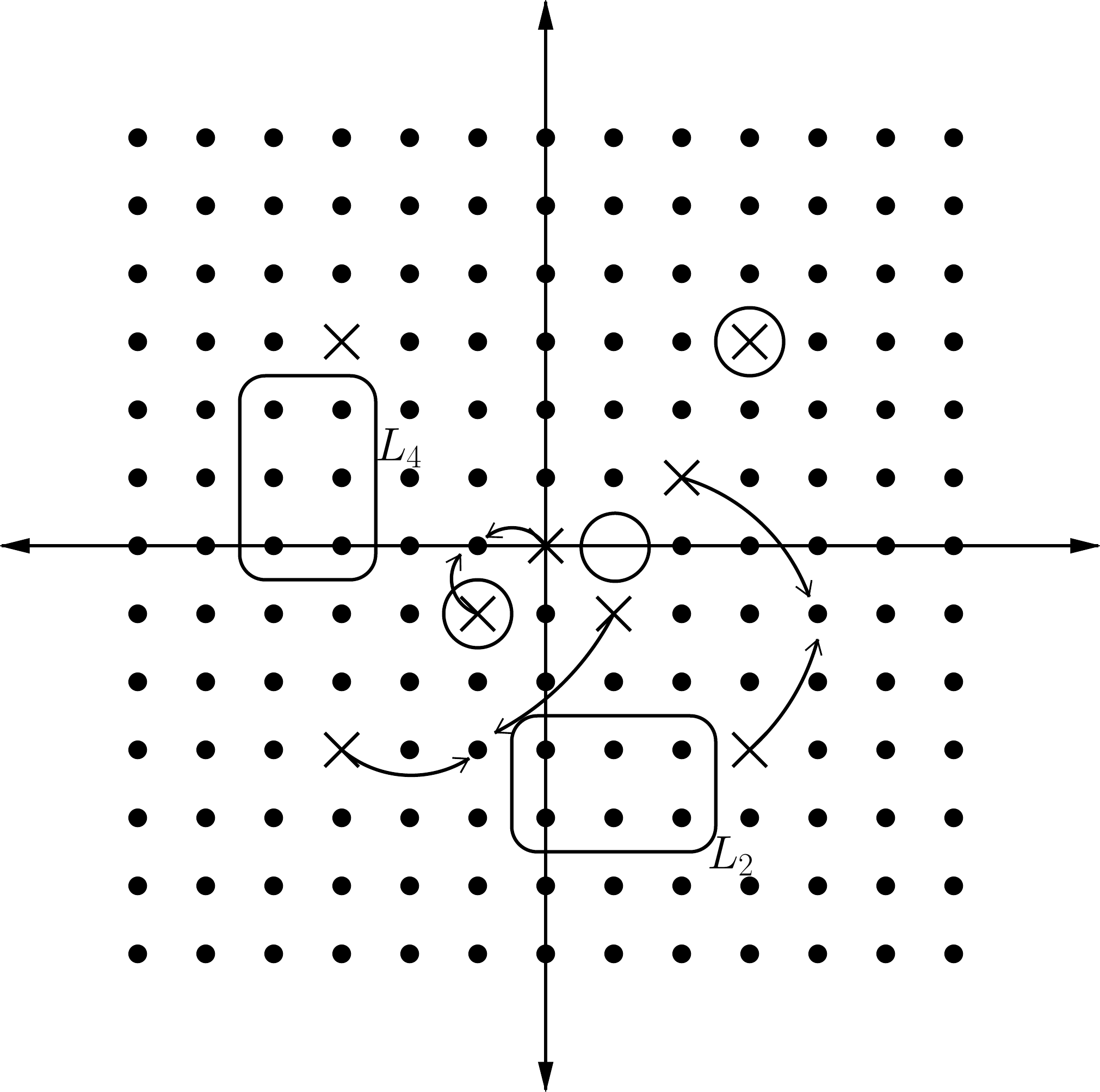}}
\caption{Case 2: at the end of round $1$.}
\label{case2-1}
\end{figure}

The revolutionaries' strategy at the beginning of round 2 is to simultaneously move revolutionaries $(0,0)$ and $(-1,-1)$ to $(-1,0)$, revolutionaries $(1,-1)$ and $(-3,-3)$ to $(-1,-3)$, and revolutionaries $(2,1)$ and $(3,-3)$ to $(4,-1)$.  Only the spy at $(-1,-1)$ can guard the first meeting, and consequently only the spy in $L_2$ can guard either of the other two meetings.  Thus the revolutionaries win.

\end{proof}

\begin{corollary} \label{maincor}
Theorem \ref{85thm} implies Theorem \ref{Z2thm}. 
\end{corollary}
\begin{proof}
The winning strategy for the revolutionaries in $\cG(\Z^{\boxtimes 2},8,5,2)$ given in Theorem \ref{85thm} does not involve movement of the revolutionaries outside of the box $[-6,6]^2$ and takes no more than $5$ moves to achieve a meeting of size $2$.  Thus a spy outside the box $[-11,11]^2$ can not help prevent a meeting of size $2$.  

We play many independent copies of this game to show that the revolutionaries have a winning strategy in $\cG(\Z^{\boxtimes 2}, r, 6 \floor{\frac{r}{8}}-1,2)$.  For each $i$ with $1 \leq i \leq \floor{\frac{r}{8}}$, eight revolutionaries are placed at positions $(30i,0) + (\pm1,\pm1)$ and $(30i,0) + (\pm3,\pm3)$.  The spies must place at most five spies in some box $(30i,0) + [-11,11]^2$, and the revolutionaries then can play there according to their winning strategy in $\cG(\Z^{\boxtimes 2},8,5,2)$ to achieve a meeting of size $2$.  Thus $\sigma(\Z^{\boxtimes 2},r, 2) \geq 6 \floor{\frac{r}{8}}$.

By Theorem \ref{productlemma}, $\sigma(\Z^{\boxtimes d}, r, 2) \geq \sigma(\Z^{\boxtimes 2},r,2) \geq 6 \floor{\frac{r}{8}}$ for $d \geq 2$.
\end{proof}

\section{Conclusion}

The authors would like to thank the editor and the referees for their many helpful comments.  We are grateful for the comprehensive suggestions of one referee that greatly simplified our proof of Theorem \ref{Z2thm}.

\end{document}